\documentclass[12pt]{amsart}
\usepackage{diagbox}
\usepackage{mathtools}
\usepackage{bbm}
\usepackage{amsmath}
\allowdisplaybreaks

\mathtoolsset{showonlyrefs}

\renewcommand{\emph}[1]{\textit{#1}}
\usepackage{enumerate,amsmath,amsthm,latexsym,amssymb}
\usepackage{color}\usepackage{graphicx}

\newcommand{\colormark}{\blue}

\definecolor{brown}{cmyk}{0, 0.72, 1, 0.45}
\definecolor{grey}{gray}{0.5}

\def\blue{\color{blue}}

\def\red{\color{red}}

\newcommand{\old}[1]{}

\newcounter{rot}

\newcommand{\ignore}[1]{}

\def\cA{{\mathcal A}}
\def\cB{{\mathcal B}}

\def\cO{{\mathcal O}}

\newcommand{\set}[1]{\left\{#1\right\}}

\def\ii_(#1,#2){i_{#1}^{#2}}

\def\a{\alpha}
\def\b{\beta}
\def\d{\delta}
\def\D{\Delta}
\def\e{\varepsilon}
\def\f{\phi}

\def\g{\gamma}

\def\z{\zeta}

\def\th{\theta}

\def\l{\lambda}
\def\m{\mu}

\def\p{\pi}

\def\r{\rho}

\def\t{\tau}
\def\om{\omega}

\def\Up{\Upsilon}

\def\1{{\bf 1}}
\def\0{{\bf 0}}
\def\Up{\Upsilon}

\newcommand{\rdup}[1]{\left\lceil #1 \right\rceil}

\def\cE{\mathcal{E}}
\def\cF{\mathcal{F}}
\def\cL{\mathcal{L}}
\def\cW{\mathcal{W}}

\def\Z{\mathbb{Z}}
\newcommand{\brac}[1]{\left( #1 \right)}

\def\E{{\bf E}}

\renewcommand{\Pr}{\operatorname{\bf Pr}}
\newcommand\bfrac[2]{\left(\frac{#1}{#2}\right)}

\newcommand{\ep}{\varepsilon}

\newtheorem{theorem}{Theorem}[section]
\newtheorem{conjecture}[theorem]{Conjecture}
\newtheorem{lemma}[theorem]{Lemma}

\newtheorem{q}[theorem]{Question}

\newtheorem{remthm}[theorem]{Remark}

\newcounter{thmtemp}

\def\omk{\omega_{k,\e}}
\def\tk{\tau_{k,\e}}

\newcommand{\nospace}[1]{}

\def\path{\operatorname{PATH}}

\newcommand{\beq}[2]{\begin{equation}\label{#1}#2\end{equation}}
\newcommand{\mult}[2]{\begin{multline}\label{#1}#2\end{multline}}

\newcommand{\sbs}{\subset}
\newcommand{\stm}{\setminus}

\def\cU{\mathcal{U}}

\newcommand{\tend}{{t_{\mathrm{end}}}}

\def\hht{\widehat{t}}
\def\apb{\approx_b}
\begin{document}
\title{Diffusion limited aggregation on the Boolean lattice}
\author{Alan Frieze}\thanks{Research supported in part by NSF grant
    CCF1013110}
\author{Wesley Pegden}\thanks{Research supported in part by NSF grant
    DMS1363136}
\date{\today}

\begin{abstract}
  In the Diffusion Limited Aggregation (DLA) process on $\Z^2$, or more generally $\Z^d$, particles aggregate to an initially occupied origin by arrivals on a random walk.  The scaling limit of the result, empirically, is a fractal with dimension strictly less than $d$.  Very little has been shown rigorously about the process, however.

  We study an analogous process on the Boolean lattice $\set{0,1}^n$, in which particles take random decreasing walks from $(1,\dots,1)$, and stick at the last vertex before they encounter an occupied site for the first time; the vertex $(0,\dots,0)$ is initially occupied.  In this model, we can rigorously prove that lower levels of the lattice become full, and that the process ends by producing an isolated path of unbounded length reaching $(1,\dots,1)$.
\end{abstract}

\maketitle
\section{Introduction}
In the classical model of Diffusion Limited Aggregation (DLA), we begin with a single particle cluster placed at the origin of our space, and then, one-at-a-time, let particles take random walks ``from infinity'' until they collide with, and then stick to, the existing cluster; when the space is not recurrent, some care is required to make this precise.

Introduced by Witten and Sander in 1981 \cite{witten}, the process is particularly natural in Euclidean space (with particles taking Brownian motions) or on $d$-dimensional lattices; in these cases, the process is empirically observed to produce structures with fractal dimensions strictly less than the dimension of the space (e.g., roughly $1.7$ for $d=2$, with slight but seemingly nonnegligible dependence on details such as the choice of underlying lattice or the precise ``sticking'' condition).

Strikingly little has been proved rigorously about the model, however.  Kesten \cite{kesten} proved an a.s.~asymptotic upper bound of $Cn^{2/3}$ on the radius of the $n$-particle cluster for the lattice $\Z^2$, for example, but no nontrivial lower bounds are known.  In particular, it is not even known rigorously that the process does not have a scaling limit with positive density.   (Eldan showed that an analogous process in the hyperbolic plane \emph{does} aggregate to positive density \cite{Eldan}).  Eberz-Wagner showed at least that the process leaves infinitely many holes \cite{holes}. For some more recent results, see Benjamini and Yadin \cite{BY}.

In this paper, we study an analogous aggregation process on the Boolean lattice $\cB=\{0,1\}^n$, which evolves at discrete times $t=0,1,\dots,$ each of which has an associated cluster $C_t$.  $C_0$ consists of just the vertex $\0=(0,\dots,0)\in \cB$.  Then, for $t>0$, $C_t$ is produced from $C_{t-1}$ by choosing a random decreasing walk $\r_t$ from $\1=(1,\dots,1)$, letting $v$ be the last vertex of the longest initial segment of $\r_t$ which is disjoint from $C_{t-1}$, and setting $C_t=C_{t-1}\cup \{v\}$.  The process terminates at the first time $\tend$ when $C_{\tend}\ni \1$.

In particular, the clusters $C_t$ grow from $\0$ by aggregation of decreasing random walks from $\1$.   Our initial motivation for considering this model was to evaluate the impact of very large dimensionality on a DLA-like process.  (An analogous motivation underlies work on percolation in the Boolean lattice; see for example \cite{Bela,FP,M}.)  We will see, however, that the Boolean lattice also allows strong rigorous (and perhaps, surprising) statements to be made about the structure of the aggregate.  In particular, let $\cL_k=\{x\in \cB\mid |x|=k\}$ denote the $k$th level of $\cB$, so that $|\cL_k|=\binom{n}{k}$. We will prove the following.
\begin{theorem}\label{t.full}
There exists $c_0>0$ such that w.h.p.\footnote{A sequence of events $\cE_n,n\geq 0$ occurs with high probability (w.h.p.) if $\lim_{n\to\infty}\Pr(\cE_n)=1$.},~for all $$k\leq k_0:=\sqrt{\tfrac{c_0n}{\log n}},$$ 
we have
  \[
  \cL_k\sbs C_{\tend}.
  \]
\end{theorem}
\begin{theorem}\label{t.dense}
  For any $K>0$, we have w.h.p.~that for all $k\leq \tfrac{n}{150K\log n}$, we have
  \[
  |\cL_k\cap C_{\tend}|\geq \brac{1-\tfrac{1}{n^K}}\cdot |\cL_k|.
  \]
\end{theorem}
\begin{theorem}\label{t.most}
  For all $\e\leq\tfrac{1}{100}$, we have w.h.p.~that for all $k<\e^3 n$,
  \[
  |\cL_k\cap C_{\tend}|\geq \brac{1-2\e}\cdot |\cL_k|.
  \]
\end{theorem}
\begin{theorem}\label{t.notall}
There is a constant $c_1$ such that for all $\e<\tfrac{1}{100}$, we have w.h.p.~that for all $$k_1=\sqrt{\tfrac{c_1n}{\log n}}\leq k<\e^3n,$$
we have
  \[
  |\cL_k\stm C_{\tend}|\geq \bfrac{(1-\r)e^{-\frac{10n}{9k}}}{2}\cdot |\cL_k|,
  \]
where 
$$\r=\max\brac{1-\bfrac {k+1}{en}^{100},1-\bfrac{1}{10e}^{10}}.$$
\end{theorem}
Thus Theorems \ref{t.full}, \ref{t.dense}, and \ref{t.most} provide progressively weaker statements as $k$ increases about the fullness of the level $\cL_k$ at the end of the process; Theorem \ref{t.notall} shows that Theorems \ref{t.full}, \ref{t.dense} and \ref{t.most} are qualitatively best-possible.  A key contrast from classical DLA is that the process does ``fill'' parts of the cube, and moreover, that this can be proved. Note also that the boundary between full and not full levels occurs w.h.p. at around $k=\sqrt{\tfrac{n}{\log n}}$.

A striking (unproved) feature of the classical DLA processes is a rich-get-richer phenomenon, where long arms of the process seem to grow at a rate significantly faster than $t^{1/d}$.  In the Boolean lattice, we observe an extreme version of this kind of runaway growth:
\begin{theorem}\label{t.path}
If $a<\frac12$ then w.h.p.~for all $k\geq n-n^a$ we have that
  \[
|\cL_k\cap C_\tend|=1.
  \]
\end{theorem}
Recall that our DLA process on $\cB$ ends once \1 is occupied; Theorem \ref{t.path} implies that $\1$ becomes occupied as the terminal vertex on an isolated path of occupied vertices whose length is at least $n^a$. 

\subsection*{Notation}
In what follows we use the notation $A_n\approx B_n$ to mean that $A_n=(1+o(1))B_n$ as $n\to\infty$ and $A_n\lesssim B_n$ to mean that $A_n\leq (1+o(1))B_n$ as $n\to\infty$; we write $A_n\apb B_n$ to mean that $A_n/B_n$ is bounded above and below by positive absolute constants as $n\to\infty$.  In some places we give expressions for integer quantities that may not be integer; in cases where we do this, it does not matter whether we round up or down.

\section{Lower levels}
In this section, we prove Theorems \ref{t.full}, \ref{t.dense}, \ref{t.most} and \ref{t.notall}. We define
\begin{equation}\label{Tcond}
\tk=\frac \e 4\binom{n}{k+1}=\omk\binom{n}{k}\quad\text{for}\quad\omk=\frac \e 4 \frac{n-k}{k+1}.
\end{equation}

Roughly speaking, we expect that at time $\tk$, the level $\cL_k$ is mostly full, while higher levels are empty enough to have little effect on the process at this time.  We will prove a sequence of lemmas confirming this general picture.  First, we establish an upper bound on the height of the cluster at a time $\tk$:

\begin{lemma}Let $\phi=1+\sqrt 3$ and $0<\e<1$.  If $k<\tfrac{n}{(1+\phi)e^\f}$, then, for all $\delta>0$, we have with probability $1-o(n^{-1})$ that
  \[
\cL_j\cap C_{\tk}=\varnothing\text{ for all }j\geq (1+\phi+\delta)k. 
  \]
\end{lemma}
\begin{proof}
Consider a fixed vertex $v$ in $\cL_{k+s}$.   If it becomes occupied by time $\tk$, then there is a sequence of times $t_{k+1}<t_{k+2}<\dots<t_{k+s}\leq \tk$ such that 

\beq{ref}{
\r_{t_s}\cap \cL_{k+s}=\{v\}\text{ and }\r_{t_i}\cap \cL_{k+i-1}=\r_{t_{i-1}}\cap \cL_{k+i-1}
}

for $i=1,\dots,s$ .

  By considering the $\binom{\tk}{s}$ possible choices of the times $t_{k+1}<\dots<t_{k+s}
\leq \tk$, the probability that each $\r_{t_i}$ satisfies the intersection conditions \eqref{ref} for $i=s,s-1,\dots$, we have that 
\begin{align}
\Pr(v\text{ is occupied at time } \tk)&\leq \binom{\omk \binom{n}{k}}{s}\prod_{i=1}^s\frac 1 {\binom{n}{k+i}}\label{vocc}\\
&\leq \bfrac{\omk\bfrac{ne}{k}^ke}{s}^s\prod_{i=1}^s\frac{(k+i)^{k+i}}{n^{k+i}}\\
&=\omk^s n^{sk-sk-s(s+1)/2}\bfrac{e^{k+1}}{k^ks}^s\prod_{i=1}^s(k+i)^{k+i}\\
&\leq\omk^s n^{-s(s+1)/2}\bfrac{e^{k+1}}{k^ks}^s(k+s)^{s(k+(s+1)/2)}\\
&=\bfrac{\omk}{s}^sn^{-s(s+1)/2}e^{(k+1)s}\brac{1+\frac{s}{k}}^{ks}(k+s)^{s(s+1)/2}\\
&\leq \bfrac{\omk}{s}^s\brac{\frac{(k+s)^{1/2}\cdot e^{1+k/(s+1)}} {n^{1/2}}}^{s(s+1)}.
\end{align}

So, multiplying by $\binom{n}{k+s}$ we see that
\begin{align}
&\Pr(\exists v\in \cL_{k+s}:v\text{ is occupied at time }\tk)\\
&\leq\binom{n}{k+s} \bfrac{\omk}{s}^s\brac{\frac{(k+s)^{1/2}\cdot e^{k/(s+1)+1}} {n^{1/2}}}^{s(s+1)}\label{anyvocc}\\
&\leq \brac{\frac{\omk}{s} \bfrac{k+s}{n}^{(s+1)/2-1-k/s}e^{2+k/s+k+s }}^s.
\end{align}
Suppose now that $k=\a n,s=\b n\geq 1$. Then the above  expression becomes
$$\brac{(\a+\b)^{\b/2} e^{\a+\b} \brac{\omk s^{-1}n^{1/n}(\a+\b)^{-1/2-\a/\b}e^{2+\a/\b}}^{1/n}}^{\b n^2}.$$
We insist that $\b\geq \a$, in which case
\beq{ass1}{
\brac{\omk s^{-1} n^{1/n}(\a+\b)^{-1/2-\a/\b}e^{2+\a/\b+\a+\b}}^{1/n}\leq 1+o(1),
} 
which implies that the expression in \eqref{anyvocc} is $o(n^{-1})$ so long as $(\a+\b)^{\b/2} e^{\a+\b}<1$.

Let $\b=\g\a$ ($\g\geq 1$). Then our requirement is that $(\a(1+\g))^{\g/2}e^{1+\g}<1$ or $\a< \frac{1}{e^{2+2/\g}(1+\g)}$. Now $e^{2+2/\g}(1+\g)$ is minimized at the solution to $\g^2=2(\g+1)$, which is $\f=1+3^{1/2}$. In summary, if $\a< \frac{1}{(1+\f)e^\f}$ then with probability $1-o(n^{-1})$ all levels above $\a(1+\f+o(1))n$ are empty at time $\omk\binom{n}{\a n}$, which gives the Lemma.
\end{proof}

Now we define $\Phi_{v,t}$ to be the fraction of (monotone) paths between $\1$ and $v$ which have at least one occupied vertex other than $v$ at time $t$.  The following Lemma implies that levels above $\cL_k$ play a small role when analyzing level $\cL_k$ at time $\tk$.
\begin{lemma}\label{l.PhiEp} For all fixed $\e>0$ and all  $\a=\frac{k}{n}<\min(\tfrac{\e^2}4,\tfrac 1 {100})$, we have
\begin{equation}\label{PhiEp}
  \Pr\brac{\exists v\in \cL_k\text{ s.t. }\Phi_{v,\tk}\geq \e}=o\brac{\tfrac 1 n}.
\end{equation}
\end{lemma}
\begin{proof}
Recall that the particle in $C_t\stm C_{t-1}$ is deposited by the decreasing walk $\r_t$.
  We fix a vertex $v\in \cL_k$, choose some $\l$ such that $k+\l\leq \tfrac n 2$, and define, for each $t=1,\dots,\tk$, a random variable $\xi_{v,t}\in [0,1]$ equal to the fraction of paths between $v$ and $\cL_{k+\l+1}$ whose interiors intersect the path $\r_t$.  Let $\cU_v$ denote the set of all $2^{n-k}$ ancestors $x>v$ of $v$.  Note that $\xi_{v,t}$ is determined by the minimum $\zeta\geq 1$ such that $\rho_t$ visits $\cL_{k+\zeta}\cap \cU_v$, and, with respect to this random variable $\zeta$, can be bounded by 
\begin{equation}\label{xifromzeta}
\xi_{v,t}\leq \xi(\zeta):=\begin{cases}{\displaystyle \sum_{i=\z}^{\l}\frac{1}{\binom{n-k}{i}}}\leq \frac{3}{2\binom{n-k}{\z}}& \zeta\leq \l\\
0 & \zeta>\l.
\end{cases}
\end{equation}

Moreover, we have for $s\leq k+\l$ that
\begin{equation}\label{zeta}
\Pr(\zeta=s)=\frac{\binom{n-k}{s}}{\binom n {k+s}}\cdot\frac{k}{k+s},
\end{equation}
since this is the probability that $\r_t$ visits $\cL_{k+s}\cap \cU_v$, and then on the next step, moves outside of $\cU_v$.  In particular, we have that
\mult{Expz}{
\E(\xi(\z))\leq\frac32\sum_{s=1}^\l \frac{\binom{n-k}{s}}{\binom{n}{k+s}}\frac k {k+s}\frac{1}{\binom{n-k}{s}}
\leq\frac32\sum_{s=1}^\l \frac{1}{\binom{n}{k+s}}\\
=\frac32
\bfrac{1+\sum_{s=2}^\l \frac{(k+2)\cdots(k+s)}{(n-k-1)\cdots(n-k-s+1)}}{\binom{n}{k+1}}
\leq
\frac 2 {\binom{n}{k+1}},
}
for $k+\l<\frac 1 2 (n-k-\l+1)$, or $k+\l<\tfrac{n+1}{3}$.

We will use the following concentration inequality for nonnegative and bounded independent random variables; we show in Appendix \ref{s.bern} that this is an easy consequence of Bernstein's inequality.
\begin{lemma}\label{l.conc}
    Let $X_1,\dots,X_{\colormark N}$ be independent random variables such that, for all $i$, $\E(X_i)\leq E$ and $X_i\in [0,C]$ almost surely. 
  Then for $S_{\colormark N}=\sum_{i=1}^{\colormark N} X_i$, $E_N=\E(S_N)$, and for all $t\leq NE$, we have that
  \[
\Pr\brac{|S_n-E_n|>t}<2\exp{\brac{-\frac{t^2/4}{NEC}}}.\qedhere
\]
\end{lemma}
Note that in the same situation, Hoeffding's inequality gives $2e^{-2t^2/NC^2}$, which, ignoring constant factors in the exponent, is always worse; the point is that we are interested in the case where $E\ll C$. And though we have stated the lemma here with the condition $t\leq NE$, one could drop it and still obtain the bound $2\exp{\brac{-\frac{t^2/2}{NEC+Ct}}}$, an analogous improvement over Hoeffding anytime $t\ll NC$.

To apply Lemma \ref{l.conc}, notice that from \eqref{xifromzeta} that $\xi(s)\leq \frac{3}{2(n-k)}$ always. If $\Xi_{v,T}=\sum_{t=1}^T \xi_{v,t}$ then $\Xi_{v,T}$ is stochastically dominated by {\colormark a} sum $Z_T=\xi(\z_1)+\xi(\z_2)+\cdots+\xi(\z_T)$ where each $\z_j$ is an independent copy of a random variable $\z$ satisfying \eqref{zeta}.  Now \eqref{Expz} implies that $T\E(\xi(\z_i))=\E(Z_T)\leq \frac{2T}{\binom{n}{k+1}}$, and thus Lemma \ref{l.conc} with $t=\th,N=\t_{k,\e},E\leq \frac{2}{\binom{n}{k+1}},C=\frac{3}{2(n-k)}$ gives that
\mult{VTcons}{
 \Pr\brac{\Xi_{v,\tk}\geq \frac{2\tk}{\binom{n}{k+1}}+\th}
\leq \Pr\brac{Z_{\tk}\geq \frac{2\tk}{\binom{n}{k+1}}+\th}\leq\\
\exp\brac{-\frac{\th^2/4}{\frac{2\tk}{\binom{n}{k+1}} \cdot \frac{3}{2(n-k)} }}=
e^{-\th^2(n-k)/(3\ep)}
=e^{- \th (n-k)/6},
}
for
\[
\th=\e-\frac{2\tk}{\binom{n}{k+1}}=\frac{\e}{2} \qquad(\text{ from \eqref{Tcond}}).
\]
Now we have that
\[
\Pr\brac{\exists v\in \cL_k,\, \Xi_{v,\tk}\geq \frac{2\tk}{\binom{n}{k+1}}+\th}\leq  \brac{\frac{en}{k}}^ke^{-\th(n-k)/6}.
\]
Writing $k=\a n$, we have that
\[
\brac{\frac{en}{k}}^ke^{-\th(n-k)/6}=e^{\ln(e/\a)\a n-\th n(1-\a)/6}=o\brac{\tfrac 1 n},
\]
for any $\a<\min(\th^2,\tfrac 1 {100})$, say.  Thus for $\a<\min(\tfrac{\e^2}4,\tfrac 1 {100})$, we have that 
\begin{equation}\label{XiEp}
\Pr\brac{\exists v\in \cL_k,\, \Xi_{v,\tk}\geq \e}=o\brac{\tfrac 1 n}.
\end{equation}

Now, by taking $\l>2\phi k$, we may assume that the levels above level $\l$ are still empty at time $\tk,$ so that $\Phi_{v,\tk}\leq \Xi_{v,\tk}$, completing the proof of the Lemma.
\end{proof}
Now we define $\Up_{v,{\colormark t}}$ for $v\in \cL_k$ to be the fraction of down-neighbors of $v$ which are unoccupied at time $t$.   By controlling $\Up_{v,t}$ and $\Xi_{v,t}$ simultaneously, we can make the behavior of the cluster with respect to $v$ sufficiently predictable.
\begin{lemma}\label{l.PhiUp}
Suppose that $0<\e\leq \frac{1}{100}$ and $k\leq \e^3n$ is fixed. Then,
\begin{equation}\label{PhiUp}
 \Pr\brac{\exists t\in [\t_{k-1,\e},\tk],\, v\in \cL_k,\,  \Phi_{v,t}+\Up_{v,t} \geq 2\e}=o\brac{\tfrac 1 n}.
\end{equation}
\end{lemma}
Lemma \ref{l.PhiUp} will be proved by induction on $k$.  Before giving the proof, we use it to prove Theorems \ref{t.full}, \ref{t.dense}, and \ref{t.most}.
\begin{proof}[Proof of Theorem \ref{t.full}]
Let $\Lambda_k$ denote the set of vertices in $\cL_k$ which are still unoccupied by particles at time $\tk$.   We fix $\e=\tfrac 1 {100}$ and apply Lemma \ref{l.PhiUp}.  Since $\tk-\t_{k-1,\e}>\tfrac 1 2 \tk$, we have for any vertex $v\in \cL_k$ that if $ \Psi=\max\set{\Phi_{v,t}+\Up_{v,t}:t\in [\t_{k-1,\e},\tk]}$,
\begin{equation}\label{vunvisited}
\Pr\brac{v\text{ not occupied by }\tk\mid \Psi\leq 2\e}
\lesssim \brac{1-\frac{1-2\ep}{\binom{n}{k}}}^{\frac12\omk\binom{n}{k}} \leq e^{-\frac13\omk}.
\end{equation}
{\bf Explanation:} For a fixed time $t$ for which $\Phi_{v,t}+\Upsilon_{v,t}\leq 2\ep$, the term $\frac{1}{\binom{n}{k}}\cdot (1-2\ep)$ is a lower bound on the probability that $\r_t$ chooses to go through $v$ on level $k$, avoids occupied vertices on the way to $v$ and then chooses an occupied vertex in level $k-1$.   Conditioning on $\Psi\leq 2\ep$ (i.e., on the $2\ep$ condition for all $t$'s simultaneously) inflates these probabilities by at most a factor of $\Pr(\Phi\leq 2\ep)^{-1}=1+o(1)$.

Thus
\begin{equation}\label{expunv}
  \E\brac{|\Lambda_k|\mid\Psi\leq 2\e}\lesssim  \binom{n}{k}e^{-\omk/3}
\end{equation}
and so by the Markov inequality,
\begin{equation}
  \label{anyvunvisited}
\Pr\brac{\Lambda_k\neq \varnothing\mid\Psi\leq 2\e}\leq \binom{n}{k}e^{-\omk/4}=o(n^{-1})
\end{equation}
as long as
\begin{equation}
  \label{anyvomcond}
\omk=\frac \e 4 \frac {n-k}{k+1}\geq 5k\log(ne/k).
\end{equation}
In particular, this holds for 
\begin{equation}
  \label{anyvworksfor}
  k\leq \sqrt{\frac{\e n}{10\log n}}
\end{equation}
and gives the desired statement (recalling that $\e=1/100$).
\end{proof}

\begin{proof}[Proof of Theorem \ref{t.dense}]
Again by the Markov inequality applied to $|\Lambda_k|$, we have
\begin{equation}
  \label{manyvunvisited}
\Pr\brac{|\Lambda_k|\geq \b \binom{n}{k}\ { \bigg|\ \Psi\leq2\e}}\lesssim \frac{e^{-\omk/3}}{\b}.
\end{equation}
This is $o(n^{-1})$, assuming that $\e=\frac{1}{100}$ and 
\begin{equation}
  \label{vworksfor}
  k\leq \frac{n}{150K\log n}, \quad\b=\frac{1}{n^K}, 
\end{equation}
for any constant $K>0$, giving the theorem.  
\end{proof}

Theorem \ref{t.most} is a consequence of the following slightly stronger statement:
\begin{lemma}\label{l.positivedensity}
For all $\e>0$, we have w.h.p.~that for all $k\leq \e^3n$, $\cL_k$ is at least $(1-2\e)$ occupied at time $\tk$.
\end{lemma}
\begin{proof}
This follows directly from Lemma \ref{l.PhiUp}.
 Indeed, if there are $\ell$ occupied sites in $\cL_k$ at time $\tk$, and $m$ edges between $\cL_{k+1}$ and occupied sites in $\cL_k$, then assuming that $\Phi_{v,t}+\Up_{v,t} \leq 2\e$ for $v\in \cL_{k+1}$, (from \eqref{PhiUp}), the degrees of vertices in $\cL_k,\cL_{k+1}$ gives that with probability $1-o(n^{-1})$,
\[
\binom{n}{k+1}\cdot (1-2\e)(k+1)\leq m\leq \ell(n-k),
\]
so that $\ell\geq (1-2\e)\binom{n}{k}$.
\end{proof}

We now prove Lemma \ref{l.PhiUp}, by induction on $k$. 


\begin{proof}[Proof of Lemma \ref{l.PhiUp}]
In particular, assuming that \eqref{PhiUp} holds for some $k$, we aim to prove that if $0<\e\leq \frac{1}{100}$ and $k\leq \e^3n$ then
\begin{equation}\label{PhiUp+}
\Pr\brac{\exists t\in [\t_{k-1,\e},\t_{k,\e}],\, w\in \cL_{k},\,  \Phi_{v,t}+\Up_{v,t} \geq 2\e}\leq \frac{2k}{n^3}.
\end{equation}
Observe that since $\Phi_{v,t}$ is increasing in $t$ and $\Up_{v,t}$ is decreasing in $t$,  \eqref{PhiUp+} can be proved by showing
\begin{align}
&0<\e\leq \frac{1}{100},k\leq \e^3n\text{ implies that } \Pr\brac{\exists w\in \cL_{k},\, \Up_{w,\tk}\geq \e}\leq \frac{k}{n^3},\label{UpOR} \\
\noalign{and}\\
&0<\e\leq \frac{1}{100},k\leq \e^3n\text{ implies that } \Pr\brac{\exists w\in \cL_{k},\, \Phi_{w,\t_{k,\e}}\geq \e}\leq \frac{k}{n^3}.\label{PhiOR} 
\end{align}
Of course, \eqref{PhiOR} follows from \eqref{PhiEp}, so we just need to show \eqref{UpOR}.   For the sake of conditioning in the induction, define the event
\[
\cA_{k,\e,n}=\brac{\forall t\in [\t_{k-1,\e},\t_{k,\e}],\, w\in \cL_{k},\,  \Phi_{v,t}+\Up_{v,t} < 2\e},
\]
so that we are aiming to prove inductively that
\[
\Pr(\cA_{k,\e,n})\geq 1-\frac{k}{n^3}.
\]
As a base case we take $k=1$ which trivially satisfies \eqref{PhiUp+}. Assume $k\geq 1$ and fix some vertex $w\in \cL_{k}$, and let $N^-_w\sbs \cL_{k-1}$ be the down-neighborhood of $w$.  If we fix a set $D\sbs N^-_w$ of size $|D|=\D$, then we have, since $k\leq \e^3n$, that 
\beq{downnbr}{
\Pr\brac{D\cap C_{\tk}=\varnothing\mid \cA_{k-1,\e,n}}\leq \brac{1-\frac{4\D}{5\binom{n}{k-1}}}^{\tk-\t_{k-1,\e}}\leq e^{-\tfrac{\D\omk}{2}}.
}
{\bf Explanation:} The first inequality arises because each path $\r_t$ for $t\in [\t_{k-1,\frac 1 {10}}+1,\tk]$ has probability $\frac{\D}{\binom{n}{k-1}}$ of intersecting $D$, and conditioned on that event, applying \eqref{PhiUp} inductively with $\e=\tfrac 1 {100}$ ensures that with probability at least $\tfrac 4 5$, a particle will occupy at least one site $v$ of $D$ after step $t$ for $\t_{k-1,\e} <t\leq \tk$ (either because $v$ was already occupied before step $t$, or because $\r_t$ deposits a particle at $v$.)  The second inequality arises because $k\leq \e^3 n$ implies that $\tk-\t_{k-1,\e}\geq \frac{\omk}{2}\binom{n}{k}$.

Thus we have that
\begin{align}
&\Pr\brac{\exists w\in \cL_{k}, D\sbs N^-_w, \text{ s.t. }|D|=\D, D\cap C_{\tk}=\varnothing}\\
&\leq \binom{n}{k}\binom{k}{\D} e^{-\D\omk/2}\leq \bfrac{ne}{k}^{k}\bfrac{ke}{\D}^\D e^{-\D\omk/2}\\
&=e^{k\log (ne/k)+\D\log (ke/\D)-\D\omk/2}\leq \frac{1}{n^3},
\end{align}
if either $k\leq n^{1/2}$ or (ii) $k>n^{1/2}$ and $\D\leq k$ and $\D\omk>4\max(k\log(ne/k),\D\log(ke/\D))$.

For (ii), recalling that $\omk=\tfrac \e 4 \cdot \frac{n-k}{k+1}$, we can take
\[
\D(n-k)>\tfrac{16}{\e}\cdot k(k+1) \log (ne/k),
\]
and
\[
n-k>\tfrac{4}{\e}\cdot k\log(ke/\D),
\]
which, for $\frac{\D}{k}\geq \e$, would follow from
\begin{equation}
\label{increasingink}
\frac{20\cdot k\log (n/k)}{n-k}<\e^2,
\end{equation}
and
\begin{equation}\label{increasingink1}
e^{1-\e\frac{n-k}{4k}}<\e,
\end{equation}
respectively.  
Both \eqref{increasingink} and \eqref{increasingink1} are satisfied when $k<\e^3 n$ and $n$ is large.
\end{proof}

Lemma \ref{l.PhiUp} is not quite strong enough to prove Theorem \ref{t.notall}.  For that purpose, we prove the following Lemma, which allows stronger statements when $k$ is linear in $n$:
\begin{lemma}
\label{l.rho}
Suppose that $k\geq k_0$. Let $t_k^\r$ be the first time when a $\r$ fraction of the vertices in $\cL_k$ are occupied. We have that
\begin{equation}\label{r_PhiUp}
\Pr\brac{\exists w\in \cL_{k+1},\,  \Up_{w,t^\r_k} \geq \tfrac 1 {10}}=o\brac{\tfrac 1 n},
\end{equation}
provided that $\e=\frac1{100},\a=\frac{k}{n}\leq \e^3$ and 
\begin{equation}\label{conditions}
\r=1-\bfrac {k+1}{en}^{100}.
\end{equation}
\end{lemma}
\begin{proof}
For any constant $K$ and sufficiently large $n$, we have that 
\[
\Pr\brac{t_k^\r> \frac 1 2 \binom{n}{k}\log\bfrac{1}{1-\r}}>1-\frac  1 {n^K}.
\]
This is because the number of vertices in $\cL_k$ that are occupied at time $t$ is dominated by the number of occupied bins when $t$ balls are placed randomly into $\binom{n}{k}$ bins. Note that the expected number of occupied bins in the latter experiment is \beq{}{\binom{n}{k}\brac{1-\brac{1-\frac{1}{\binom{n}{k}}}^t}\approx (1-(1-\r)^{1/2})\binom{n}{k} \leq  \frac{\r}{2}\binom{n}{k}\text{ when }t=\frac 1 2 \binom{n}{k}\log\bfrac{1}{1-\r}.
}
Note also that the number of occupied boxes is highly concentrated. This can be verified through a simple application of McDiarmid's inequality, see \cite{Bern}.

In particular, 
\beq{wes}{
t^\r_{k}-\t_{k-1,\e}\geq H_{\r,k}:=\frac {\log(1/(1-\r))} 2\binom{n}{k}
}
with probability at least $\brac{1-\tfrac 1 {n^K}}$. This follows from the fact that $\r\geq 1-e^{-\e/4}$.  

From \eqref{conditions} we also have that
\[
H_{\r,k}=50\log\bfrac{en}{k+1}\cdot \binom{n}{k},
\]
and so we have $\t_{k-1,\e}+H_{\r,k}\leq \t_{k,\e}$ for sufficiently large $n$; see \eqref{Tcond}.  In particular, we can apply Lemma \ref{l.PhiUp} in the entire range $[\t_{k-1,\e},\t_{k-1,\e}+H_{\r,k}]$.

To do this, we fix some vertex $w\in \cL_{k+1}$, and let $N^-_w\sbs \cL_{k}$ be the down-neighborhood of $w$.  If we fix a set $D\sbs N^-_w$ of size $|D|=\D$, then we have for $k\leq\e^3n$ that
\begin{multline}
\Pr\brac{D\cap C_{t^\r_{k}}=\varnothing\mid t^\r_k\geq \t_{k-1,\e}+H_{\r,k}}\leq \brac{1-\frac{1}{n^K}}^{-1}\brac{1-\frac{4\D}{5\binom{n}{k}}}^{H_{\r,k}}\\
\leq 2e^{-\tfrac{\D\cdot \log(1/(1-\r))} {5}}
=2(1-\r)^{\D/5}.
\end{multline}
{\bf Explanation:} 
We repeat the argument for \eqref{downnbr} and multiply by $(1-n^{-K})^{-1}$ to account for conditioning on $t^\r_k\geq \t_{k-1,\e}+H_{\r,k}$.

Thus we have that
\begin{align}
&\Pr\brac{\exists w\in \cL_{k+1}, D\sbs N^-_w, \text{ s.t. }|D|=\D=(k+1)/10, D\cap C_{\tk}=\varnothing\mid t^\r_k\geq \t_{k-1,\e}+H_{\r,k}}\\
&\leq \binom{n}{k+1}\binom{k+1}{\D}\times 2(1-\r)^{\tfrac {\D} 5}
\leq 2\bfrac{ne}{k+1}^{k+1}\bfrac{(k+1)e}{\D}^\D (1-\r)^{\D/5}.\\
&=  2e^{(k+1)\log (ne/k+1)+\D\log ((k+1)e/\D)+(\D/5)\log(1-\r)}
  =o\brac{\frac{1}{n}},
\end{align}
if $\D=(k+1)/10$ and $\D\log(\tfrac 1{1-\r})\geq 10\max\set{(k+1)\log (ne/(k+1)),\D\log((k+1)e/\D)}$.

For these cases, we require that
\begin{equation}
\label{r_increasingink}
\r\geq1-\bfrac {k+1}{e n}^{100}
\end{equation}
and
\begin{equation}
\r\geq 1-\bfrac{\D}{e(k+1)}^{10}=1-\bfrac{1}{10e}^{10},
\end{equation}
respectively, both of which follow from our choice of $\rho$.
\end{proof}

We are now ready to prove Theorem \ref{t.notall}.
\begin{proof}[Proof of Theorem \ref{t.notall}]
We apply Lemma \ref{l.rho} with $\e=\tfrac 1 {100}$ and $\r$ satisfying \eqref{conditions}. Condition on the event $\cF=\{\forall v\in \cL_k, \,\Up_{v,t^\r_k} < \tfrac 1 {10}\}$.  Let $O_k=\cL_{k}\cap C_{t_k^\r}$ be the set of occupied vertices in $\cL_{k}$, so that $|O_k|=\rdup{\r \binom{n}{k}}$.  Fix a vertex $v\in U_k=\cL_{k}\stm O_k$ and let $N^+_v$ be its neighborhood in $\cL_{k+1}$.  For each $w\in N^+_v$, we define events $\cE_w^1$ and $\cE_w^2$, respectively by
\begin{enumerate}
\item Every path $\r_t$, $t>t^\r_k$ which contains $w$ avoids both $O_k$ and $v$, and
\item The first path $\r_t$, $t>t^\r_k$ which intersects $w$ and then hits $O_k\cup\set{v}$ hits $O_k$ and not $v$,
\end{enumerate}
and set $\cE_w=\cE_w^1\cup \cE_w^2.$
Let $\cA_v$ be the event that $\cE_w$ occurs for each $w\in N^+_v$.   Observe that if $\cA_v$ occurs then vertex $v$ remains unoccupied on termination. Moreover, if we fix some set $W_0\sbs N^+_v$, then the events $\cE_w$ for $w\in W_0$ are conditionally independent, given the event that $W_0=\cW:=\{w\in N^+_v\text{ such that } \neg \cE_w^1 \text{ occurs}\}$.  (Indeed, given that $\neg \cE_w^1$ occurs, we know that at least one path goes through $w$.  Moreover, the event $\cE_w^2$ depends on just the first path with this property, and the choice this path makes below $w$ is independent of choices made by paths not going through $w$.)  Now, for any choice of $W_0$ and any $w\in W_0$, we have
\[
\Pr(\cE_w\mid \cF,\cW=W_0)\geq 1-\frac{1}{\tfrac 9 {10}k},
\]
since $|(O_k\cap N^-_w)\cup \{v\}|\geq \tfrac 9 {10}k$, and $\cE_w^2$ implies that the first path through $w$ choosing among $(O_k\cap N^-_w)\cup \{v\}$ chooses $v$.  Now using the conditional independence of the $\cE_w$'s given $W_0$, we have that 
\beq{vend1}{
\Pr(\cA_v\mid \cF,\cW=W_0) \geq\brac{1-\frac{10}{9k}}^{n-k}\geq p:=e^{{-1/\a}}.
}
Finally, since this is true for any fixed $W_0$, we have that
\beq{vend}{
\Pr(\cA_v\mid \cF) \geq\brac{1-\frac{10}{9k}}^{n-k}\geq p.
}

 It follows from \eqref{vend} that on termination, conditioning on $\cF$, there are in expectation at least $(1-\r)p\binom{n}{k}$ vertices of $U_k$ that remain unoccupied at the end of the process.

Let now $Z_k$ denote the number of $v$ such that $\cA_v$ occurs. Now $Z_k$ is determined by at most $(k+1)\binom{n}{k+1}$ random choices viz.~the paths from {\bf 1} to $\cL_{k+1}$ that give rise to a first visit to a vertex of $\cL_{k+1}$ that continues on to $O_k$. More precisely, we partition the paths from {\bf 1} to {\bf 0} according to which member of $\cL_{k+1}$ they visit. $Z_k$ is determined by an independent choice of a path from each part of the partition followed by a choice of vertex in $\cL_k$. Changing one of these choices, changes $Z_k$ by at most one and so applying McDiarmid's inequality we get
\beq{success}{
\Pr\brac{Z_k\leq \frac{1-\r}{2}p\binom{n}{k}}\leq \exp\set{-\frac{(1-\r)^2p^2\binom{n}{k}^2}{2(k+1)\binom{n}{k+1}}}\leq \exp\set{-\frac{\bfrac{k}{en}^{100}e^{-2n/k}\bfrac{n}{k}^k}{n}}.
}
Now
$$\log\brac{\frac{\bfrac{k}{en}^{100}e^{-2n/k}\bfrac{n}{k}^k}{n}}=(k-100)\log(n/k)-100-\log n-\frac{2n}{k}\geq 2\log n$$
if $k_1\leq k\leq \e^3n$. This proves the Theorem.
\end{proof}

\section{Long path}
In this section we prove Theorem \ref{t.path}
\subsection{Setup}
We begin our proof by fixing certain parameters $a,b,c$. Recall that $n^a$ is the length of the path that we prove exists.  $n^b$ will be a bound on the expected value of a level at a certain time, and the exponent $c$ will occur in error bounds in our concentration analysis.  

Our proof will require that these parameters satisfy the following constraints:
\begin{enumerate}
\item $a<1-2c$. This is needed to ensure that the probability in \eqref{prob} is $o(n^{-1})$ as claimed.
\item $2c<1-a$. This is needed to ensure that the RHS of \eqref{summary} is $o(\m_0)$.
\item $a<2c$. This is needed to ensure that $\d_j$ in \eqref{ineq1} to be $o(1)$.
\item $a+b<1$. This is needed to ensure that the LHS of \eqref{last} is o(1).
\item $a<\frac12$. This is also needed to ensure that the LHS of \eqref{last} is o(1).
\item $a>b$.  This is needed in \eqref{noA1b}.
\end{enumerate}
We choose $a$ as large as possible here. So we take
$$a=\frac12-\e,\;b=\frac12-2\e,\;c=\frac14-\frac{\e}3$$
for some arbitrarily small $\e>0$.  

We then let
\begin{equation}\label{lisna}
\ell=n^a\text{ and }k=n-\ell
\end{equation}
and assume that $\ell$ is an integer.
We let $\cO_{j,t}=C_t\cap \cL_j$, the set of occupied vertices on level $j$ at time $t$. 

A considerable difficulty facing our proof of Theorem \ref{t.path} is that we do not understand the ``intermediate'' behavior of the cube; that is, our Theorems \ref{t.full}, \ref{t.dense}, \ref{t.most}, \ref{t.notall} lose their bite well below level $\tfrac n 2$, say.  Thus the proof must be agnostic to the behavoir of the process in the middle layers of the cube.  One natural idea to handle this would be to to assume a ``worst-case'' behavior for the intermediate levels of the cube; say, that level $k$ becomes full while levels $k+j$ $(j\geq 1)$ are still empty, and show that even in this scenario, a path of length nearly $n-k$ will still grow, for sufficiently large $k$.  However, the DLA process is not monotone in a clean way, preventing us from arguing directly that having level $k$ full while higher levels are empty is truly a worst-case scenario from the standpoint of the probability that a long isolated path reaches $\bf 1$.

Instead, we proceed by defining a stopping time.  We run the DLA process on the empty cube, until time $\t_0$ when there first exists $0\leq j^*<j_0$ such that  
\beq{defzeta}{
|\cO_{k+j^*,\t_0}|\geq \zeta(j^*,\m_0).
}
Here  $\m_0$ and $j_0$ are parameters which will chosen later, and we define
\beq{}{
\z(j,\m_0)=\frac{\binom{\m_0}{j_0+1}}{j_0\binom{\m_0}{j_0-j} \eta(j)},
}
and
\[
\eta(j)=\prod_{s=0}^{j-1}\binom{n}{\ell-s}.
\]
$\m_0$ will be an upper bound estimate for the time after $\t_0$ when we can expect the process to end, and $j_0\approx \sqrt{2\ell}$ will be the level from which we show the isolated path with grow.  We will see in \eqref{noA1b} that $\zeta(j_0-1,\m_0)<1$, so that the stopping time $\t_0$ always occurs.

Roughly speaking, by beginning our analysis from this stopping time, we begin from a situation where we have some (carefully chosen) useful bounds on the sizes of levels, which makes an analysis of the remainder of the process possible.

For the purpose of analyzing the growth of the DLA process in expectation, it is useful to allow the process to continue past the point when vertex $\1$ becomes occupied.  To do this, we extend the DLA process past time $\tend$ by letting $\Theta_t$ be the number of particles stuck ``above'' $\1$.   In particular, $\Theta_t=\max\set{0,t-\tend}$, and occupancies of vertices $v\in \cB$ at times $t>\tend$ are the same as at time $\tend$.

Now we let $X_{j,t}=|\cO_{k+j,\t_0+t}|$ for $0\leq j,t$ and let $Y_{j,t}=\Theta_t^2+\sum_{r\geq j^*+j}X_{r,t}$. (It would be natural to replace $\Theta_t^2$  with $\Theta_t$ here, but using $\Theta_t^2$---or any fast-enough growing function of $\Theta_t$---ensures that the following recurrence for $Y_{j,t}$ will not be broken by the cases where $t\gg \tend$.)  Then we have that for $j\geq1$,
\begin{align}
\E(X_{j,t}-X_{j,t-1}\mid X_{j-1,t-1})&\leq \frac{X_{j-1,t-1}}{\binom{n}{\ell-j+1}},\quad t\geq 1.\label{re1x}\\
\E(Y_{j,t}-Y_{j,t-1}\mid Y_{j-1,t-1})&\geq \frac{Y_{j-1,t-1}}{\sum_{r\geq j^*+j}\binom{n}{\ell-r+1}}
\approx \frac{Y_{j-1,t-1}}{\binom{n}{\ell-j^*-j+1}}
,\quad t\geq 1.\label{re2x}
\end{align}
{\bf Explanation}\\
The RHS of \eqref{re1x} is the probability that a particle chooses an occupied position on level $k+j$. It is an upper bound for the increase because it does not account for the particle being blocked higher up in the cube.

For the middle term in \eqref{re2x}, observe that there are $Y_{j-1,t-1}$ occupied vertices among the $\sum_{r\geq j^*+j}\binom{n}{\ell-r+1}$ vertices at or above level $j-1$; thus the middle term gives the probability that a randomly chosen vertex from $\r_t\cap \bigcup_{r\geq j^*+j} \cL_{\ell-r+1}$ is occupied, and the occurrence of this event implies that $Y_{j,t}$ increases by one.    This explains the first inequality.

Removing the conditioning in \eqref{re1x}, \eqref{re2x} we obtain for $j\geq 1$,
 \begin{align}
\E(X_{j,t}-X_{j,t-1})&\leq \frac{\E(X_{j-1,t-1})}{\binom{n}{\ell-j+1}},\quad t\geq 1.\label{re1}\\
\E(Y_{j,t}-Y_{j,t-1})&\gtrsim \frac{\E(Y_{j-1,t-1})}{\binom{n}{\ell-j^*-j+1}}
,\quad t\geq 1.\label{re2}
\end{align}

The recurrences \eqref{re1}, \eqref{re2} yield upper and lower bounds as on the expectations of $X_{j,t},Y_{j,t}$, which will be analyzed in Section \ref{occop}. 

To prove that a path grows from $j_0$, we will first show that after
\begin{equation}\label{m1}
\m_1:=\m_0-\om_1\binom{n}{\ell-j_0}
\end{equation}
\newcommand{\tfin}{t_{\rm fin}}
steps for 
$$\om_1=\log^2 n,$$ 
we will have that w.h.p.~$Y_{j_0,\t_0+\m_1}\gtrsim n^b$ and $X_{j_0-1,\t_0+\m_1}\lesssim \om_1n^{b}$.  Observe that this implies that for the minimum $t_{\rm fin}$ for which $Y_{j_0,\tfin}=1$, we have $X_{j_0-1,\tfin}\lesssim \om_1n^{b}$, and that we have that $|\cO_{k+j_0,\tfin}|=1$.    In particular, we will prove that the DLA process can quickly produce a path from to $\1$ after $\tfin$; that $X_{j_0-1,t}$ does not increase quickly after $\tfin$, and that the small value of $X_{j_0-1,t}$ for $t$ near $\tfin$ implies that no particles stick at $j_0$ while the path to $\1$ is being created.

\subsection{Choice of $\m_0,\m_1,j_0$}
In this section we define $\m_0,\m_1,j_0$ and compute various quantities associated with them for later use.  In particular, we let
\begin{equation}\label{xi}
\xi(j,t)=\frac{\binom{\m_0}{j_0+1}}{\eta(j)}\cdot \frac{\binom{t}{j}}{\binom{\m_0}{j_0}}=\zeta(j,\m_0)\cdot \frac{j_0\binom{\m_0}{j_0-j}\binom{t}{j}}{\binom{\m_0}{j_0}}.
\end{equation}
Note from \eqref{lisna} that
\beq{noA1}{
\z(j_0-1,\m_0)=\xi(j_0-1,\m_0)\frac{\m_0-j_0+1}{j_0^2\m_0}.
}
This, together with lines \eqref{j0=}, \eqref{qaz3} and \eqref{43}, below, will imply then that
\beq{noA1b}{
\z(j_0-1,\m_0)=O\bfrac{n^b}{n^a}=o(1).
}

Roughly speaking, $\xi(j,t)$ is an approximate target for comparison with $|\cO_{k+j,\t_0+t}|$.
In particular, we will choose $\mu_0,j_0$ and prove that
\begin{enumerate}[{\bf P1:}]
\item \hspace{1in}$\E(X_{j_0,\m_1})\lesssim \xi(j_0,\m_0)\approx n^b$ -- see \eqref{qaz3} and \eqref{nn2}.
\item \hspace{1in}$\E(X_{j_0-1,\m_1})\lesssim \xi(j_0-1,\m_0)\approx2e^{\frac43}n^b$ -- see \eqref{qaz2}.
\item \hspace{1in}$\E(Y_{j_0,\m_1})\gtrsim \frac{\xi(j_0,\m_0)}{j_0}$ -- see \eqref{Yeq}.
\end{enumerate}

We choose 
\beq{omega}{ 
\om=(1-a)\log n
}
and then $j_0$ by 
\[
j_0=\min\set{j:j(j+3)\geq 2\ell+\frac{4\ell}{\om-1}}.
\]
Now $j(j+3)-(j-1)(j+2)=2j+2$ and so we have that
\beq{j0def}{
2\ell+\frac{4\ell}{\om-1}\leq j_0(j_0+3)\leq 2\ell+\frac{4\ell}{\om-1}+2j_0+2.
}
Thus
\beq{j0=}{
j_0^2=2\ell+\frac{4\ell}{\om-1}+\th_0,
}
where $|\th_0|\leq 3j_0$.

Next we prove an asymptotic estimate for $\eta(j)$.
\begin{lemma}\label{l.etaest}
  If $\ell\gg j$, $n-\ell+j\gg 1$, and $\ell=o(n^{1/2})$, then 
  \begin{equation}
    \label{etaest}
    \eta(j)\approx \bfrac{n}{\ell}^{j(\ell-(j-1)/2)} (2\p\ell)^{-\frac12j} e^{\ell j+3\ell^2j/2n+\e_aj-\e_bj^2},
  \end{equation}
  where
  \begin{equation}
    \label{etaeps}
    \e_a=\frac{j}{2\ell}+O\brac{\bfrac{j}{\ell}^2},\quad \e_b=\frac{j}{6\ell}+\frac{j^2}{24\ell^2} +O\brac{\bfrac{j}{\ell}^3}.
  \end{equation}
\end{lemma}
\begin{proof}
We let
\begin{equation}
  \label{superfac}
  \f(\ell)=\prod_{r=1}^\ell r!
\end{equation}
be the superfactorial function.  It is known that
\begin{equation}\label{superfacest}
  \f(\ell)\approx C_1\ell^{\frac12\ell^2+\ell+\frac5{12}}\,e^{-\frac34\ell^2-\ell}\,(2\p)^{\frac12\ell}
\end{equation}
for some absolute constant $C_1>0$.  (See, for example, Adamchik \cite{Sloane}. We use the asymptotic expression for the Barnes function $G(z)$ on page 2. Note also that $\f(\ell)=G(\ell+2)$.)

We need to estimate $\frac{\f(\ell-x)}{\f(\ell)}$ where $x=O(\ell^{1/2})$. In preparation we observe that if $x=O(\ell^{1/2})$ then
\begin{equation}
\brac{1-\frac{x}{\ell}}^{\ell-x}=\exp\set{-x+\sum_{i=2}^\infty \frac{x^{i}}{i(i-1)\ell^{i-1}}} =e^{-x+\e_ax},
\end{equation}
where $\e_a=\sum_{i=2}^\infty\frac{x^{i-1}}{i(i-1)\ell^{i-1}} =\frac{x}{2\ell}+O\brac{\bfrac{x}{\ell}^2}$, and
\begin{equation}
\brac{1-\frac{x}{\ell}}^{\frac12(\ell-x)^2}=\exp\set{-\frac12x\ell+\frac{3}{4}x^2- \sum_{i=3}^\infty\frac{x^{i}}{i(i-1(i-2))\ell^{i-2}}} =e^{-\frac12x\ell+\frac34x^2-\e_bx^2},
\end{equation}
where $\e_b=\sum_{i=3}^\infty\frac{x^{i-2}}{i(i-1(i-2))\ell^{i-2}}=\frac{x}{6\ell}+\frac{x^2}{24\ell^2} +O\brac{\bfrac{x}{\ell}^3}$.

Thus, if $x=O(\ell^{1/2})$ then
\begin{align}
\frac{\f(\ell-x)}{\f(\ell)}&\approx\frac{ \brac{1-\frac{x}{\ell}}^{\frac12(\ell-x)^2}\brac{1-\frac{x}{\ell}}^{\ell-x} \brac{1-\frac{x}{\ell}}^{\frac{5}{12}} \ell^{\frac12(\ell-x)^2+\ell-x+\frac5{12}}\, e^{-\frac34(\ell-x)^2-(\ell-x)}\,(2\p)^{\frac12(\ell-x)}} {\ell^{\frac12\ell^2+\ell+\frac5{12}}\,e^{-\frac34\ell^2-\ell}\,(2\p)^{\frac12\ell}}\\
&= \brac{1-\frac{x}{\ell}}^{\frac12(\ell-x)^2}\brac{1-\frac{x}{\ell}}^{\ell-x} \brac{1-\frac{x}{\ell}}^{\frac{5}{12}} \ell^{-\ell x+\frac12x^2-x} e^{\frac32\ell x-\frac34x^2+x}(2\p)^{-\frac12x}\\
&\approx e^{-x-\frac12x\ell+\frac34x^2+\e_ax-\e_bx^2}\times \ell^{-\ell x+\frac12x^2-x} e^{\frac32\ell x-\frac34x^2+x}(2\p)^{-\frac12x}\\
&= \ell^{-\ell x+\frac12x^2-x} e^{\ell x+\e_ax-\e_bx^2}(2\p)^{-\frac12x}
.\label{ratiox}
\end{align}
Observe that if $s^3=O(m^2)$ for some $m\to\infty$ with $n$ then
\beq{binom}{
\binom{m}{s}= \frac{m^s}{s!}\exp\set{-\frac{s^2}{2m}+O\bfrac{s^3}{m^2}}.
}

So, if $\ell\gg j,n-\ell+j\gg1$ and $\ell=o(n^{1/2})$ then
\beq{xival}{
\eta(j)\approx\frac{n^{j(\ell-(j-1)/2)} e^{-3\ell^2j/2n}\f(\ell-j)}{\f(\ell)}\approx \bfrac{n}{\ell}^{j(\ell-(j-1)/2)} (2\p\ell)^{-\frac12j} e^{\ell j+3\ell^2j/2n+\e_aj-\e_bj^2},
}
as desired.
\end{proof}
As a consequence, we have: 
\begin{lemma}\label{l.xiapprox}
If $\m_0\geq \m_1\gg \ell^2$ then 
\begin{equation}
  \xi(j_0,\m_0)\approx  \frac{\m_0^{j_0+1}}{(j_0+1)!\eta(j_0)}\approx \frac{(\m_0 e)^{j_0+1}(2\p)^{\frac12j_0}\ell^{j_0(\ell-\frac12j_0+1)}}{\sqrt{2\p j_0}(j_0+1)^{j_0+1}n^{j_0(\ell-\frac12(j_0-1))}e^{(\ell -\frac13+\e_0)j_0}},
\end{equation}
where $\e_0=o(1)$.
\end{lemma}
\begin{proof}
Using \eqref{j0=}, we compute
\begin{align}
\e_aj_0-\e_b^2j_0^2&=\frac{j_0^2}{2\ell}-\frac{j_0^3}{6\ell}-\frac{j_0^4}{24\ell^2} +O(j_0^{-1})\\
&=\brac{1+O\bfrac{1}{\om}}-\brac{\frac{j_0}3+O\bfrac{j_0}{\om}}+o(1)\\
&=-\brac{\frac13-\e_0}j_0\text{ where }\e_0=o(1).\label{rreecc}
\end{align}
The lemma now follows by using \eqref{binom} to deal with $\binom{\m_0}{j_0+1}$ and Stirling's approximation and \eqref{xival} to deal with $\eta(j_0)$. The factor $e^{3\ell^2j_0/2n}$ can be absorbed into the $\e_0j_0$ term.
\end{proof}

Now choose $\m_0$ as 
\begin{align}
\m_0&:=\frac{1}{e}\bfrac{(j_0+1)^{1+\frac1{j_0}}n^{\ell-\frac12(j_0-1)} e^{\ell-\frac13+\frac4{3j_0}}j_0^{\frac1{2j_0}}n^{\frac{b}{j_0}}} {(2\p)^{ \frac{j_0-1}{2j_0}} \ell^{\ell-\frac12j_0+1}}^{\frac{j_0}{j_0+1}} \\  
&\approx \frac{j_0}{e^{\frac43}\sqrt{2\p\ell}}\bfrac{n^{\ell-\frac12(j_0-1)} e^{\ell}}
{\ell^{\ell-\frac12(j_0-1)}}^{\frac{j_0}{j_0+1}}\label{valm0}\\
&=\bfrac{n}{\ell}^{\ell+o(\ell)}.\label{valm00}
\end{align}
Observe that with this choice, we have from Lemma \ref{l.xiapprox},
\beq{qaz3}{
\xi(j_0,\m_0)\approx n^b.
}
We now compare $\m_0$ and $\m_1$.
\begin{lemma}
\beq{m0m1}{
\m_1\geq \m_0(1-e^{-j_0}).
}
\end{lemma}
\begin{proof}
Evaluating the exponents in \eqref{valm0}, we see from \eqref{j0=} that
\beq{obs1}{
\frac{(\ell-\frac12(j_0-1))j_0}{j_0+1}-(\ell-j_0)= \frac{\frac12(j_0^2+j_0)-\ell+j_0}{j_0+1}=\frac{2\ell}{(\om-1)(j_0+1)}-O(1).
}
 In particular this says that 
\beq{valmu0}{
\m_0\gtrsim \frac{j_0}{e^{\frac43}\sqrt{2\p\ell}}\bfrac{n}{\ell}^{\ell-j_0+ \frac{2\ell}{(\om-1)(j_0+1)}-O(1)}e^{\frac{\ell j_0}{j_0+1}}.
}
It follows from this that
\begin{multline}\label{f1}
\frac{\m_0}{\binom{n}{\ell-j_0}}\geq 
\frac{\m_0\ell^{\ell-j_0}e^{j_0-\ell}}{n^{\ell-j_0}}\gtrsim \frac{j_0}{e^{\frac43}\sqrt{2\p\ell}}\bfrac{n}{\ell}^{\frac{2\ell}{(\om-1)(j_0+1)}-O(1)} e^{j_0-\frac{\ell}{j_0+1}}\geq\\
 \frac{j_0n^{-O(1)}}{e^{\frac43}\sqrt{2\p\ell}} \brac{\bfrac{n}{\ell}^{\frac{2}{\om}}e^{-1}}^{\frac{\ell}{j_0+1}} e^{j_0} = \frac{j_0n^{-O(1)}e^{\frac{\ell}{j_0+1}+j_0}}{e^{\frac43}\sqrt{2\p\ell}}\geq e^{j_0},
\end{multline}
and the lemma follows.
\end{proof}
We now compare $\xi(j_0,\m_0)$ and $\xi(j_0-1,\m_0)$.
\begin{lemma}
\beq{43}{
\frac{\xi(j_0-1,\m_0)}{\xi(j_0,\m_0)}=\Theta(1).
}
\end{lemma}
\begin{proof}
  We have that
\[
\frac{\xi(j_0-1,\m_0)}{\xi(j_0,\m_0)}= \frac{\binom{\m_0}{j_0-1}\binom{n}{\ell-j_0+1}}{\binom{\m_0}{j_0}}.
\]
Using \eqref{binom} and applying Stirling's formula to $(\ell-j_0+1)!$, we get that
  \begin{multline}
    \frac{\xi(j_0-1,\m_0)}{\xi(j_0,\m_0)}\approx \frac{j_0(ne)^{\ell-j_0+1}}{\sqrt{2\p\ell}\m_0(\ell-j_0+1)^{{\ell-j_0+1}}}\\
\approx \frac{j_0(ne)^{\ell-j_0+1}}{\sqrt{2\p\ell}(\ell-j_0+1)^{{\ell-j_0+1}}}\cdot \frac{e^{\frac43}\sqrt{2\p\ell}}{j_0}\cdot \bfrac{\ell^{\ell-\frac12(j_0-1)}}{n^{\ell-\frac12(j_0-1)} e^{\ell}}^{\frac{j_0}{j_0+1}},
  \end{multline}
where at the end we have used \eqref{valm0}.
Now
\begin{align}
(\ell-j_0+1)^{\ell-j_0+1}&=\ell^{\ell-j_0+1} \brac{1-\frac{j_0-1}{\ell}}^{\ell-j_0+1}\\
&=\ell^{\ell-j_0+1}\exp\set{-(\ell-j_0+1)\brac{\frac{j_0-1}{\ell} +\frac{(j_0-1)^2}{2\ell^2}+O\bfrac{j^3}{\ell^3}}}\\
&\approx\ell^{\ell-j_0+1}e^{-j_0}.
\end{align}

So, we can write 
\begin{align}
\frac{\xi(j_0-1,\m_0)}{\xi(j_0,\m_0)}&\approx e^{\frac43}\bfrac{\ell}{ne}^{(\ell-\frac12(j_0-1))\frac{j_0}{j_0+1}-(\ell-j_0+1)} e^{j_0-\frac12(j_0-1)\frac{j_0}{j_0+1}}\\
&= e^{\frac43} e^{\frac{j_0(j_0+3)}{2(j_0+1)}} \bfrac{\ell}{ne}^{\frac{2\ell}{(\om-1)(j_0+1)}+O(1)}.
\end{align}
Our choice of $\omega$ implies that
$$\frac{2\ell}{\om-1}\log\bfrac{ne}{\ell}=\frac{2\ell(\om+1)}{\om-1}.$$
And then we see that
\begin{align}
&\frac{\xi(j_0-1,\m_0)}{\xi(j_0,\m_0)}\\
&\approx e^{\frac43}\cdot \exp\set{\frac{1}{2(j_0+1)}\brac{j_0(j_0+3)-\frac{2\ell}{\om-1} \log\bfrac{ne}{\ell}+O(\log n)}}\\
&\approx e^{\frac43}\cdot \exp\set{\frac{1}{2(j_0+1)}  \brac{j_0(j_0+3)-\frac{2\ell(\om+1)}{\om-1}}}\\
&=\Theta(1).
\end{align}
Here we have used \eqref{j0=}.
\end{proof}
\subsection{Expected occupancies}\label{occop}
Now we analyze the random variables $X_{j,t}$ and $Y_{j,t}$ in expectation.  The next lemma helps us deal with the recurrences \eqref{re1}, \eqref{re2}.
\begin{lemma}\label{rec1}
Let $\a_{-1}=\b_{-1}=1$.
\begin{enumerate}[{\bf (a)}]
\item Let $x_{j,t}$ satisfy (i) $x_{0,t}\leq \a_0+t$, (ii) $x_{j,0}\leq \a_j$, (iii) $x_{j,t}-x_{j,t-1}\leq \b_{j-1}x_{j-1,t-1}$ for $j\geq 1$, where $\b_j\geq 0$ for $j\geq 0$. Then when $j\geq 1$ we have
\beq{lemref}{
x_{j,t}\leq \sum_{i=0}^{j+1}\a_{j-i}\binom{t}{i}\prod_{s=j-i}^{j-1}\b_s.
}
\item Let $y_{j,t}$ satisfy (i) $y_{j^*,t}\geq \a_{j^*}$, (ii) $y_{j,0}\geq 0$ for $j>j^*$, and (iii) $y_{j,t}-y_{j,t-1}\geq \b_{j-1}y_{j-1,t-1}$ for $j>j^*,t\geq 1$.  Then for $j\geq j^*$ we have 
\beq{lemrefy}{
y_{j,t}\geq \a_{j^*}\binom{t}{j-j^*}\prod_{s=j^*}^{j-1}\b_s.
}
\end{enumerate}
\end{lemma}
\begin{proof}
(a) Now we have
$$x_{1,t}\leq\a_1+ \b_0\sum_{\t=1}^t(\a_0+\t-1)=\a_1+\a_0\b_0t+\b_0\binom{t}{2}.$$
So equation \eqref{lemref} is true for $j=1$. Assume inductively that it is true for $j-1$ where $j\geq 2$. Then
\begin{align}
x_{j,t}&\leq\a_j+\b_{j-1}\sum_{\t=1}^tx_{j-1.\t-1}\\
&\leq\a_j+\b_{j-1}\sum_{\t=1}^t\sum_{i=0}^{j}\a_{j-1-i}\binom{\t-1}{i} \prod_{s=j-i-1}^{j-2}\b_s\\
&=\a_j+\b_{j-1}\sum_{i=0}^j\a_{j-1-i}\binom{t}{i+1} \prod_{s=j-i-1}^{j-2}\b_s\\
&=\a_j+\sum_{i=0}^j\a_{j-1-i}\binom{t}{i+1} \prod_{s=j-i-1}^{j-1}\b_s\\
&=\a_j+\sum_{i=1}^{j+1}\a_{j-i}\binom{t}{i} \prod_{s=j-i-1}^{j-1}\b_s.
\end{align}
(b) We have $y_{j^*,t}\geq \a_{j^*}$ and so equation \eqref{lemrefy} is true for $j=j^*$. Assume inductively that it is true for $j-1$ where $j>j^*$. Then,
\begin{align}
y_{j,t}&\geq\b_{j-1}\sum_{\t=1}^tx_{j-1.\t-1}\\
&\geq\b_{j-1}\sum_{\t=1}^t\a_{j^*}\binom{\t-1}{j-1-j^*} \prod_{s=j^*}^{j-2}\b_s\\
&=\a_{j^*}\binom{t}{j-j^*}\prod_{s=j^*}^{j-1}\b_s.
\end{align}
\end{proof}

\subsubsection{Upper Bound}\label{upperbound}
To use Lemma \ref{rec1} for an upper bound on $\E(X_{j,\t_0+t}),j\geq 0$ we use the definition of the stopping time $\t_0$ to define
\[
\a_j=|\cO_{k+j,\t_0}|\leq \zeta(j,\m_0)=\frac{\binom{\m_0}{j_0+1}}{j_0\binom{\m_0}{j_0-j}\eta(j)}\text{ and }\b_j=\frac{1}{\binom{n}{\ell-j}}\text{ for }j\geq 0.
\]
Thus, for $0\leq j\leq j_0$ and $\m_1\leq t\leq \m_0$,
\begin{align}
\E(X_{j,\t_0+t})&\leq \sum_{i=0}^{j+1}|\cO_{k+j-i,\t_0}|\binom{t}{i} \prod_{s=j-i}^{j-1}\frac{1}{\binom{n}{\ell-s}}\\
&\leq \sum_{i=0}^{j+1}\z(j-i,\m_0)\binom{t}{i} \prod_{s=j-i}^{j-1}\frac{1}{\binom{n}{\ell-s}}\\
&=\sum_{i=0}^{j+1}\frac{\binom{\m_0}{j_0+1}}{j_0\binom{\m_0}{j_0-j+i}\eta(j-i)} \binom{t}{i} \prod_{s=j-i}^{j-1}\frac{1}{\binom{n}{\ell-s}} \label{check1}\\
&\lesssim \frac{\binom{\m_0}{j_0+1}}{\binom{\m_0}{j_0}} \binom{t}{j} \prod_{s=0}^{j-1}\frac{1}{\binom{n}{\ell-s}}\label{check3}\\
&=\xi(j,t).
\end{align}
To go from \eqref{check1} to \eqref{check3} we let $u_i$ denote the summand in \eqref{check1} and observe that
\beq{urat}{
\frac{u_{i+1}}{u_i}=\frac{t-i}{i+1}\cdot\frac{j_0-j+i+1}{\m_0-j_0+j-i}.
}
This implies that
\[
\frac{u_j}{u_i}\geq \brac{1-\frac{\om_1\binom{n}{\ell-j_0}}{\m_1}}^{j_0}=1-o(1),
\]
whenever $i<j$. Here we have used \eqref{m0m1}. Furthermore, \eqref{urat} implies that $u_{j+1}\lesssim u_j$ for $j\leq j_0$. This verifies \eqref{check3}.

So:
\begin{lemma}\label{Xbounds}
  We have
\beq{nn2}{
\E(X_{j_0,\t_0+\m_1})\lesssim \xi(j_0,\m_1)\approx \xi(j_0,\m_0)\approx n^b
}
and 
\beq{qaz2}{
\E(X_{j_0-1,\t_0+\m_1})\lesssim \xi(j_0-1,\m_0)\approx\xi(j_0,\m_0)\times\Theta(1)=\Theta(n^b).
}
\end{lemma}
\begin{proof}
  These come from \eqref{qaz3}, \eqref{m0m1}, \eqref{43} and \eqref{check3}.
\end{proof}

\subsubsection{Lower Bound}

To use Lemma \ref{rec1} for a lower bound on $\E(Y_{j,t}),j\geq 0$ we use $y_{j,t}=Y_{j,t}$ for $j\geq j^*$, and take
\beq{low1}{
\a_{j^*}=|\cO_{k+j^*,\t_0}|\geq \z(j^*,\m_0)\text{ and }\b_j=\frac{1}{\binom{n}{\ell-j}},
}
to get:
\begin{lemma} \label{l.Ybound}For $j\geq 0$,
\begin{equation}\label{low2}
\E(Y_{j,t})\geq
\z(j^*,\m_0)\binom{t}{j-j^*}\prod_{s=0}^{j-j^*-1}\frac{1}{\binom{n}{\ell-j^*-s}}= \frac{\binom{\m_0}{j_0+1}}{j_0\eta(j)}\frac{\binom{t}{j-j^*}}{\binom{\m_0}{j_0-j^*}}.
\end{equation}
\end{lemma}
\subsection{Concentration}  We can obtain w.h.p.~upper bounds on the sizes of sets $\cO_{j_0+j,t}$ by applying Markov's inequality to the random variables $X_{j,t}$. In this section, we obtain suitable w.h.p.~lower bounds on the random variables $Y_{j,t}$.
Let
\begin{equation}\label{NjLj}
N_j=\binom{n}{\ell-j}\text{ and }L_j=\frac{\binom{\m_0}{j_0+1}}{j_0 \eta(j)\binom{\m_0}{j_0-j^*}}\text{ for }j\geq 0.
\end{equation}
Observe from Lemma \ref{l.Ybound} that we have
\begin{equation}
  \label{Y1bound}
  \E(Y_{j+1,t})\geq \frac{\binom{t}{j+1-j^*}L_{j}}{N_j}.
\end{equation}
We will establish lower concentration of the level sizes inductively, starting from level $j^*+1$.  For each level $j>j^*$, there will be a time $t_j$ past which we have a good w.h.p.~lower bound on the level size, which can then be used inductively for the next level.

We define $\hht_j,j^*+1\leq j\leq j_0-1$ by 
\begin{equation}
  \hht_j=\min\set{t:\frac{L_j\binom{t}{j-j^*}}{N_j}\geq n}.
\end{equation}
And then define $t_j,j^*+1\leq j\leq j_0-1$ by 
\begin{equation}
  t_j=\max\set{\hht_j,\om_2t_{j-1}},\quad\om_2=n^{c}.
\end{equation}
We also let $t_{j^*}=\t_0$. 

Our definition ensures that $t_j\ll t_{j+1}$ for all $j^*<j<j_0-1$.  The following Lemma shows that the $t_j$'s don't grow to large.
\begin{lemma}
For $j^*+1\leq j\leq j_0-1$,
\beq{summary}{
\bfrac{n}{\ell}^{\ell-O(j_0)}\leq t_j\leq \om_2^{j-j^*}\bfrac{\ell}{n}^{j_0/3}\m_0\ll \m_0.
}
\end{lemma}
\begin{proof}
We will first need to bound $\hht_j$ from above and below. Thus we estimate
\begin{align}
\frac{L_j}{N_j}&=\frac{\binom{\m_0}{j_0+1}}{j_0\eta(j+1)\binom{\m_0}{j_0-j^*}}\\
&\approx \frac{\m_0^{j^*+1}(j_0-j^*)!}{j_0(j_0+1)!\eta(j+1)}\\
&\approx \frac{(j_0-j^*)!}{j_0(j_0+1)!}\brac{\frac{j_0}{e^{\frac43+o(1)} \sqrt{2\p\ell}}\bfrac{n^{\ell-\frac12(j_0-1)} e^{\ell}}
{\ell^{\ell-\frac12(j_0-1)}}^{\frac{j_0}{j_0+1}}}^{j^*+1}\\
&\hspace{2in}\times \bfrac{\ell}{n}^{(j+1)(\ell-j/2)} (2\p\ell)^{\frac12(j+1)} e^{-\ell(j+1)-3\ell^2(j+1)/2n-\e_a(j+1)+\e_b(j+1)^2}\\
&\leq \bfrac{\ell}{n}^{\ell(j-j^*)+\frac12(j_0j^*-j^2){\red -}O(j_0)}.
\end{align}
Thus,
\beq{tjb}{
t_j\geq \hht_j\geq \bfrac{n}{\ell}^{\ell-\frac{j^2-j_0j^*}{2(j-j^*)}-O(j_0)} \geq \bfrac{n}{\ell}^{\ell-O(j_0)},
}
since clearly $\frac{j^2-j_0j^*}{j-j^*}\leq j_0$.

On the other hand, because $\hht_j$ is large we can write 
\mult{}{
n\approx \frac{L_j\binom{\hht_j}{j-j^*}}{N_j}\approx \bfrac{\hht_j}{\m_0}^{j-j^*}\frac{L_j\binom{\m_0}{j-j^*}}{N_j} =\bfrac{\hht_j}{\m_0}^{j-j^*}\frac{\xi(j_0,\m_0)\eta(j_0)}{j_0\eta(j+1)}\approx\\ \bfrac{\hht_j}{\m_0}^{j-j^*}\frac{n^b}{j_0}\prod_{i=j+1}^{j_0-1}\binom{n}{\ell-i}.}
We see from this that
\mult{}{
\frac{\m_0}{\hht_j} \gtrsim\brac{\frac{n^{b-1}}{j_0}  \prod_{i=j+1}^{j_0-1}\bfrac{n}{\ell}^{\ell-i}}^{1/(j-j^*)} =\\ \bfrac{n^{b-1}}{j_0}^{1/(j-j^*)}\bfrac{n}{\ell}^{(j_0-j-1)(\ell- \frac12(j_0+j))/(j-j^*)}\geq \bfrac{n}{\ell}^{j_0/{\red 3}}.
}
Consequently,
\beq{tju}{
t_j\leq \om_2^{j-j^*}\hht_j\leq \bfrac{\ell}{n}^{j_0/3}\m_0,
}
which completes the proof of the lemma.
\end{proof}
Our next task is to obtain a high probability lower bound on the random variables $Y_{j^*+1,t}$. Define, for $j^*\leq j\leq j_0-1$,
\beq{ineq1}{
\d_{j}=\frac{j-j^*}{\om_2}
}
We define $\cE_j$ to be the event that there is a $\t\in [t_j,\m_0]$ such that $Y_{j,\t}<(1-\d_j)L_j\binom{\t}{j-j^*}$.
\begin{lemma}\label{l.ej}For all $j^*\leq j\leq j_0-1$, we have
  \beq{eq0}{\Pr\brac{\cE_j}\leq \frac{j-j^*}{n^2}.
    }
Moreover, we have that w.h.p.
\beq{Yeq}{
Y_{j_0,\m_1}\gtrsim  \frac{\xi(j_0,\m_0)}{j_0}.
}
\end{lemma}
\begin{proof}
  We prove \eqref{eq0} by induction.
The base case $j=j^*$ is trivial because all we assume is that $Y_{j^*,t}\geq L_{j^*}=\z(j^*,\m_0)$ for $t\geq \t_0$. 

Assume now that $j^*+1\leq j+1\leq j_0-1$.  We write 
\[\tilde \cL_j=\bigcup_{j'\geq j} \cL_{j'}\]
and 
\[\tilde N_j=|\tilde \cL_j|=\sum_{j'\geq j}\binom{n}{\ell-j'}.\]
We define a new random variable $Z_{ j+1,t}=\th_{j+1,t_j}+\th_{j+1,t_{j}+1}+\dots+\th_{j+1,t}$, where the $\th_{j+1,\t}$'s are independent $\set{0,1}$ random variables where 
\beq{probtheta}{
\E(\th_{j+1,\t})=\frac{\rdup{(1-\d_j)L_j\binom{\t}{j-j^*}}}{\tilde N_j}\gtrsim \frac{(1-\d_j)L_j\binom{\t}{j-j^*}}{N_j}
} 
We will define these variables so that  
\beq{notEj}{
\neg \cE_j \quad\text{implies}\quad Y_{j+1,t}\geq Z_{j+1,t}\text{ for }t\in [t_{j},\m_1]. 
}
For each $\t\geq t_j$, we define $S_{j,\t}$ to be the lexicographically first subset of $\bar \cL_j$ among subsets of size $\rdup{(1-\d_j)L_j\binom{\t}{j-j^*}}$ for which a maximum possible number of vertices at level $j$ are occupied.  (In particular, $\neg \cE_j$ implies that $S_{j,\t}$ is full.) We let $\th_{j+1,\t}$ be the indicator random variable for the event that the path $\r_\t$ used at step $\t$ intersects $S_{j,\t}$. Observe that the $\th_{j+1,\t}$'s are independent for $\t\geq t_j$ and also that \eqref{probtheta} and \eqref{notEj} hold.

Now for $t\geq t_{j+1}$ we have
\begin{align}
\E(Z_{j+1,t})&\gtrsim \frac{(1-\d_j)L_j}{N_j}\sum_{\t=t_{j}}^{t-1}\binom{\t}{j-j^*}\\
&=\frac{(1-\d_j)L_j}{N_j}\brac{\binom{t}{j-j^*+1}-\binom{t_j+1}{j-j^*+1}}\\
&\geq \frac{(1-\d_j)L_j\binom{t}{j-j^*+1}}{N_j}\brac{1-\bfrac{t_j+1}{t}^{j-j^*+1}}\\
&\geq\frac{(1-\d_j)L_j\binom{t}{j-j^*+1}}{N_j}\brac{1-\bfrac{1}{\om_2}^{j-j^*+1}}\\
&\geq \frac{\brac{1-\d_j-\frac{1}{2\om_2}}L_j\binom{t}{j-j^*+1}}{N_j}
\end{align}
And applying Hoeffding's theorem, we see that 
\mult{prob}{
\Pr\brac{\exists t\in [t_{j+1},\m_1]:Z_{j+1,t}\leq \frac{(1-\d_{j+1})L_{j}\binom{t}{{j+1-j^*}}}{N_j}}\leq \sum_{t=t_{j+1}}^{\m_1} \exp\set{-\frac{L_{j}\binom{t}{{j+1-j^*}}}{10\om_2^2N_j}}\\
\leq \m_1\exp\set{-\frac{n}{10\om_2^2}}=\bfrac{n}{\ell}^{\ell+o(\ell)} \exp\set{-\frac{n^{1-2c}}{10}}.
}
Thus 
$$\Pr(\cE_{j+1})\leq \Pr(\cE_{j})+\bfrac{n}{\ell}^{\ell+o(\ell)} \exp\set{-\frac{n^{1-2c}}{10}} \leq \frac{j-j^*+1}{n^2},$$
completing the inductive proof of \eqref{eq0}. 

Evaluating just the $t=\m_1$ term from \eqref{prob} with $j=j_0-1$, $\d_{j_0}=n^{-c/4}$ gives
\begin{multline}
\Pr\brac{Y_{j_0,\m_1}\leq \frac{(1-\d_{j_0})L_{j_0-1}\binom{\m_0}{{j_0-j^*}}}{N_{j_0-1}}}\leq \Pr\brac{Z_{j_0,\m_1}\leq \frac{(1-\d_{j_0})L_{j_0-1}\binom{\m_0}{{j_0-j^*}}}{N_{j_0-1}}}+\Pr(\cE_{j_0-1})\\\leq \exp\set{-\frac{L_{j_0-1}\binom{\m_0}{{j_0-j^*}}}{10n^{c/2}N_{j_0-1}}}+o(n^{-1}) =\exp\set{-\frac{\xi(j_0,\m_0)}{10j_0n^{c/2}}}+o(n^{-1})\\\leq \exp\set{-\frac{n^b}{20n^{(c+a)/2}}}+o(n^{-1})
=o(n^{-1}).
\end{multline}
(Notice we only proved concentration for one value of $t$ in the $j_0$ case, as opposed to an interval as in \eqref{prob}.)
So we have that w.h.p.
\[
Y_{j_0,\m_1}\gtrsim \frac{L_{j_0-1}\m_0^{j_0-j^*}}{N_{j_0-1}(j_0-j^*)!}\approx  \frac{\binom{\m_0}{j_0+1}\m_0^{j_0-j^*}}{j_0\binom{\m_0}{j_0-j^*}\eta(j_0)(j_0-j^*)!} \approx \frac{\binom{\m_0}{j_0+1}}{j_0\eta(j_0)}=\frac{\xi(j_0,\m_0)}{j_0},\]
giving \eqref{Yeq}.
\end{proof}

So now let $\t_0^*\leq \m_1$ be the first time that level $k+j_0$ is non-empty. Equation \eqref{Yeq} shows that $\t_0^*$ exists w.h.p. Now consider the next $\om_1\binom{n}{\ell-j_0}$ particles. We argue that w.h.p.~these particles create an isolated path from level $k+j_0$ to the top. 

Observe that we have w.h.p.~that
\beq{good}{
|\cO_{k+j_0,\t_0^*}|=1\text{ and }|\cO_{k+j_0-1,\t_0^*}|\leq n^b\log n.
}
For the second bound we have used the Markov inequality and \eqref{qaz2} from Lemma \ref{Xbounds}.
Note that we have $\cO_{j,\t_0^*}=\varnothing$ for $j>k+j_0$. Next let $\t^*_i,i \geq 1$ be the time when the first particle occupies $\cL_{k+j_0+i}$. We observe that
\beq{stopgap}{
\Pr\brac{\t^*_{i}-\t^*_{i-1}\geq\om_1\binom{n}{k+j_0+i-1}}\leq \brac{1-\frac{1}{\binom{n}{k+j_0+i-1}}}^{\om_1\binom{n}{k+j_0+i-1}}\leq e^{-\om_1}.
}
We observe next that for $i\geq 0$ we have
\beq{total}{
\om_1\sum_{r=i}^{n-k-j_0}\binom{n}{k+j_0+r}\leq 2\om_1\binom{n}{k+j_0+i}
}
This implies that for $i\leq n-k-j_0$ we that w.h.p.,
\beq{sumsig}{
\t^*_i\leq \t_0^*+2\om_1\binom{n}{k+j_0}\ll \m_0.
}
For the final inequality we used \eqref{f1}. 

In particular, $\tend-\t_0^*=\t^*_{n-k-j_0}-\t_0^*\ll \m_0$.
Finally, let us consider the probability that $|\cO_{j,\tend}|\geq 2$ for some $i\geq0,j\geq k+j_0$.  At a fixed time $t\in [\t_0^*,\m_0]$, the probability that a particle lands on level $k+j_0$ is at most $\frac{n^b\log n}{\binom{n}{k+j_0-1}}$ and the probability that a particle at time $t\in (\t^*_i,\tend]$ lands at level $k+j_0+i$ ($i>0$) by colliding with the \emph{first} particle which landed at level $k+j_0+i-1$ is at most $\frac{1}{\binom{n}{k+j_0+i-1}}$.  Thus, using \eqref{stopgap}, \eqref{total} and \eqref{sumsig}, the probability the particle at time $t\in [\t_0^*,\tend]$ becomes the second particle to occupy a level $j\in [k+j_0,n]$ is at most
\pushQED{\qed}
\beq{last}{
2\om_1 \binom{n}{k+j_0}\frac{n^{b}\log n}{\binom{n}{k+j_0-1}}+ \sum_{i=1}^{n-k-j_0-1}2\om_1\binom{n}{k+j_0+i}\frac{1}{\binom{n}{k+j_0+i-1}}+o(1) =o(1).\qedhere
}
\popQED

\section{Further Questions}
In some sense, our theorems characterize the beginning and end of the process under consideration.  Understanding the behavior of the process in the middle of the cube seems like a major challenge.  On the other hand, it is likely to be a prerequisite for an understanding of some basic parameters of the model.  For example, from empirical evidence, the following seems likely:
\begin{conjecture}
$\tend=o(2^n)$.
\end{conjecture}

Of course an extremely natural target is the following:
\begin{q}
How large can the parameter $a$ be in Theorem \ref{t.path}?
\end{q}

There are also some interesting modifications of the model to consider.  For example, what happens if the random walks are not monotone?  In the Boolean lattice this may seem a bit unnatural.  It may be interesting to consider hypercubes $[m]^n$, in which case the behavior of the process relative to the relationship between $m$ and $n$ can be explored.

\appendix
\section{Concentration inequality}\label{s.bern}
In this section we derive Lemma \ref{l.conc}.  Recall the statement:
\begin{lemma}[Lemma \ref{l.conc}]
    Let $X_1,\dots,X_N$ be independent random variables such that, for all $i$, $\E(X_i)\leq E$ and $X_i\in [0,C]$ almost surely. Then for $S_N=\sum_{i=1}^N X_i$, $E_N=\E(S_N)$, and for all $t\leq NE$, we have that
  \[
\Pr\brac{|S_N-E_N|>t}<2\exp{\brac{-\frac{t^2}{4NEC}}}.
  \]
\end{lemma}
   This is an immediate consequence of Bernstein's inequality (see, e.g., \cite{Bern}):
   \begin{lemma}[Bernstein]\label{l.bern}
     Let $X_1,\dots,X_N$ be independent random variables and $V,C$ be constants such that
     \begin{equation}
       \label{V}
       \sum_{i=1}^N \E(X_i^2)\leq V
     \end{equation}
     and for all $q\geq 3$,
     \begin{equation}
       \label{Cq}
       \sum_{i=1}^N \E(\max(X_i^q,0))\leq \frac{q!}{2}VC^{q-2}.
     \end{equation}
     Then  for $S_N=\sum_{i=1}^N X_i$, $E_N=\E(S_N)$, we have that
     \pushQED{\qed} 
     \[
     \Pr\brac{S_N-E_N>t}<\exp{\bfrac{-t^2}{2(V+Ct)}}.\qedhere
     \]
     \popQED
   \end{lemma}

\begin{proof}[Proof of Lemma \ref{l.conc}]
  In the setting of Lemma \ref{l.conc}, the conditions of Lemma \ref{l.bern} hold for the random variables $X_i$ (as well as for the random variables $-X_i$) by taking $C$ as given, and taking $V=NCE$, since $0\leq X_i\leq C$. So we have that
  \[
  \Pr\brac{S_N-E_N>t}<\exp{\bfrac{-t^2/2}{NCE+Ct}}\leq \exp{\bfrac{-t^2/2}{2NCE}},
  \]
  assuming $t\leq EN$.  The analogous statement holds for $-S_N$ also, giving the Lemma.
  \end{proof}

\begin{thebibliography}{99}
\bibitem{BY} I. Benjamini and A. Yadin, Upper bounds on the growth rate of Diffusion Limited Aggregation, Arxiv:1705.06095v1, 2017.
\bibitem{Bern}  S. Boucheron, G. Lugosi and P. Massart, Concentration Inequalities: A Nonasymptotic Theory of Independence, Oxford University Press, Oxford, UK, 2013.
\bibitem{Bela} B. Bollob\'as and Y. Kohayakawa, On Richardson’s model on the hypercube, {\em Combinatorics, geometry
and probability} (1993), Cambridge Univ. Press, Cambridge, 1997, pp. 129–137.
\bibitem{Sloane} V. Adamchik, Symbolic and Numeric Computations of the Barnes Function, {\em Computer Physics Communications}, 157 (2004) 181-190.
\bibitem{holes} D.M. Eberz-Wagner, Discrete growth models, Ph.D Thesis (1999).
\bibitem{Eldan}  R. Eldan, Diffusion-limited aggregation on the hyperbolic plane, {\em The Annals of Probability} 43 (2015) 2084-2118.
\bibitem{FP}  J.A. Fill, and R. Pemantle, Percolation, first-passage percolation and covering times for Richardson's model on the n-cube, {\em The Annals of Applied Probability} (1993) 593-629.
\bibitem{kesten} H. Kesten, How long are the arms in DLA?, {\em Journal of Physics A} 20 (1987) 29-33.
\bibitem{M}  A. Martinsson, Unoriented first-passage percolation on the n-cube, {\em The Annals of Applied Probability} 26 (2016) 2597-2625.
\bibitem{witten} T. Witten and L. Sander, Diffusion-limited aggregation, {\em Physical review B} 27 (1983) 5686-5697.
\end{thebibliography}
\end{document}